\RequirePackage[l2tabu, orthodox]{nag} 
\documentclass[11pt,a4paper]{article}

\title{Projection method for eigenvalue problems of\\linear nonsquare matrix pencils}
\author{Keiichi Morikuni\thanks{Faculty of Engineering, Information and Systems, University of Tsukuba, 1-1-1 Tennodai, Tsukuba, Ibaraki 305-8573, Japan. Email:~morikuni.keiichi.fw@u.tsukuba.ac.jp. The work was supported in part by the Japan Society for the Promotion of Science (No.~16K17639 and No.~20K14356) and Hattori Hokokai Foundation.}}
\date{}

\usepackage{amsmath, amsfonts, amssymb, amsthm}
\usepackage[all, warning]{onlyamsmath} 
\usepackage{graphicx}
\usepackage[
	bookmarks=false,
	bookmarksnumbered=true,%
	bookmarksopen=false,	
	bookmarkstype=toc,%
	pdfauthor={Keiichi Morikuni},%
	pdfdisplaydoctitle={true},
	pdftitle={Projection method for eigenvalue problems of linear nonsquare matrix pencils},%
	setpagesize={false},%
	urlcolor={red},
	pdfstartview={FitH -32768}
]{hyperref}
\hypersetup{pdfpagemode=UseNone} 

\usepackage {booktabs}
\usepackage{tikz}
\usepackage{arydshln}

\RequirePackage{xcolor}[2.11]
\colorlet{siaminlinkcolor}{green!50!black}
\colorlet{siamexlinkcolor}{red!50!black}
\colorlet{siamreviewcolor}{black!50}

\if@hidelinks
\hypersetup{ hidelinks = true }
\else
\hypersetup{
	colorlinks = true,
	allcolors = siaminlinkcolor,
	urlcolor = siamexlinkcolor,
}
\fi

\RequirePackage[capitalize,nameinlink]{cleveref}[0.19]

\usepackage{enumitem}
\usepackage[T1]{fontenc}
\usepackage{exscale, fix-cm, lmodern, textcomp}
\usepackage[top=1in, bottom=1in, left=1in, right=1in]{geometry}
\usepackage{algorithm, algorithmic}
\usepackage[subrefformat=parens]{subcaption}

\hypersetup{pdfpagemode=UseNone}

\usepackage{upgreek}

\usepackage{autonum} 

\usepackage{algorithm, algorithmic}

\theoremstyle{plain}
\newtheorem{theorem}{Theorem}[section]
\newtheorem{lemma}[theorem]{Lemma}
\theoremstyle{remark}
\newtheorem{remark}[theorem]{Remark}

\numberwithin{equation}{section}
\makeatletter
	
	\@addtoreset{equation}{section}
\makeatother

\makeatletter
  
  \@addtoreset{algorithm}{section}
\makeatother

\makeatletter
  
  \@addtoreset{table}{section}
\makeatother

\newfont{\bg}{cmr9 scaled\magstep4}
\newcommand{\bigzerol}{\smash{\lower1.0ex\hbox{\bg 0}}}
\newcommand{\bigzerou}{\smash{\hbox{\bg 0}}}

\allowdisplaybreaks

\begin{document}
\maketitle

\begin{abstract}
Eigensolvers involving complex moments can determine all the eigenvalues in a given region in the complex plane and the corresponding eigenvectors of a regular linear matrix pencil.
The complex moment acts as a filter for extracting eigencomponents of interest from random vectors or matrices.
This study extends a projection method for regular eigenproblems to the singular nonsquare case, thus replacing the standard matrix inverse in the resolvent with the pseudoinverse.
The extended method involves complex moments given by the contour integrals of generalized resolvents associated with nonsquare matrices.
We establish conditions such that the method gives all finite eigenvalues in a prescribed region in the complex plane.
In numerical computations, the contour integrals are approximated using numerical quadratures.
The primary cost lies in the solutions of linear least squares problems that arise from quadrature points, and they can be readily parallelized in practice.	
Numerical experiments on large matrix pencils illustrate this method.	
The new method is more robust and efficient than previous methods, and based on experimental results, it is conjectured to be more efficient in parallelized settings.
Notably, the proposed method does not fail in cases involving pairs of extremely close eigenvalues, and it overcomes the issue of problem size.
\end{abstract}

\section{Introduction}
Consider the computation of all finite eigenvalues of a linear matrix pencil~$z B - A \in \mathbb{C}^{m \times n}$, $z \in \mathbb{C}$, $A$, $B \in \mathbb{C}^{m \times n}$ in a prescribed simply connected open set $\Omega \subset \mathbb{C}$
\begin{align}
	A \boldsymbol{x} = \lambda B \boldsymbol{x}, \quad \boldsymbol{x} \in \mathbb{C}^n \setminus \lbrace \boldsymbol{0} \rbrace, \quad \lambda \in \Omega
	\label{eq:evp}
\end{align}
and the corresponding eigenvectors.
The matrix pencil $zB-A$ is said to be regular if $m = n$ and~$\det (z B - A)$ is not identically equal to zero for all $z \in \mathbb{C}$; otherwise, it is singular.
This study focuses on singular cases.
Such problems~\eqref{eq:evp} arise in linear differential-algebraic equations~\cite{KunkelMehrmann2006}, the eigenstate computations of semiconductor quantum wells~\cite{Alharbi2010}, collocation method for approximating the eigenfunctions of the Hilbert--Schmidt operator~\cite[Chapter~12]{FasshauerMcCourt2015}, and supervised dimensionality reduction~\cite{MatsudaMorikuniSakurai2018IJCAI,MatsudaMorikuniImakuraYeSakurai2020}.

A stable and well-established method for computing eigenvalues of linear matrix pencils involves the use of is to use the QZ algorithm~\cite{MolerStewart1973SJNA}.
This method reduces a matrix pencil to a (quasi) triangular form (in which $2 \times 2$ blocks along the diagonal may exist) using a pair of unitary matrices.
Extensions of Kublanovskaya's algorithm~\cite{Kublanovskaya1983} combined with the QZ algorithm and staircase algorithm have been proposed in \cite{Dooren1979, Wilkinson1979}.
These methods use unitary equivalence transformation to determine the Kronecker structure of a linear matrix pencil~$z B - A$, including eigenvalues, Jordan block sides, and minimal indices.
A sophisticated implementation of these methods is the generalized upper triangular (GUPTRI) algorithm~\cite{DemmelKaagstroem1993a, DemmelKaagstroem1993b, GUPTRI, MCSToolbox}.

A new wave of eigensolver's development emerged, initiated by the scaling method~\cite{Goedecker1995,Goedecker1999} for electronic structure analysis, the Sakurai--Sugiura method~\cite{SakuraiSugiura2003JCAM}, and the FEAST algorithm~\cite{Polizzi2009}.
In this class of projection methods, a complex moment consisting of a resolvent filters out undesired eigencomponents and extracts the desired ones in a pseudo-random matrix whose columns are supposed to have eigencomponents corresponding to the eigenvalues of interest.
Thus, methods of this kind project a regular matrix pencil onto the eigenspace associated with eigenvalues in a prescribed region and give the eigenvalues and the corresponding eigenvectors of a regular matrix pencil.
The complex moment results from a contour integral of a resolvent matrix.
In numerical computations, the integral is approximated by a quadrature rule. 
Each quadrature point produces a linear system of equations to solve.
Each linear system can be solved independently, and hence, this kind of methods can be efficiently implemented in parallel.

This framework includes many versions, and it has been applied to other types of eigenproblems: 
\begin{itemize}
	\item Hankel-type (Petrov--Galerkin-type) approach and its block variant for generalized eigenproblems~\cite{SakuraiSugiura2003JCAM,IkegamiSakuraiNagashima2010}, polynomial eigenproblems~\cite{AsakuraSakuraiTadanoIkegamiKimura2010JJIAM}, and nonlinear eigenproblems~\cite{AsakuraSakuraiTadanoIkegamiKimura2009JSIAMLett},  
	\item Rayleigh-Ritz-type approach and its block variant for generalized eigenproblems~\cite{SakuraiTadano2007,IkegamiSakurai2010} and nonlinear eigenproblems~\cite{YokotaSakurai2013JSIAMLett},
	\item FEAST (subspace iteration) for generalized eigenproblems~\cite{Polizzi2009,KestynPolizziTang2016SISC,TangPolizzi2014SIMAX} and its operator analogue~\cite{HorningTownsend2020SINUM},		
	\item Beyn's method for nonlinear eigenproblems~\cite{Beyn2012LAA,BarelKravanja2016JCAM}, and
	\item block Arnoldi approach for generalized eigenproblems~\cite{ImakuraDuSakurai2014,ImakuraSakurai2017RIMS}, cf.~\cite{FangSaad2012SISC}.
\end{itemize}
Another perspective of this class of projection methods is in the filter diagonalization, leading to the design of rational filters~\cite{Murakami2008,AustinTrefethen2015SISC}.
See \cite{ImakuraDuSakurai2016JJIAM} for mutual relationships among projection methods of this kind.

Their advantage is in its hierarchal parallelization.
This feature is suitable for the heterogeneous architectures of modern computers. The typical bottleneck in parallelization is the data transfer among memories in different layers and among processors.
The following hierarchical procedures can be efficiently executed in parallel in the projection methods:
\begin{itemize}
	\item internal eigencomputations for partitioned regions in the top layer~\cite{MiyataDuSogabeYamamotoZhang2009TJSIAM,MaedaSakurai2015}, 
	\item solutions of linear systems arising from quadrature points in the middle layer~\cite{GuttelPolizziTangViaud2015SISC}, and 
	\item computations of each linear solve in the bottom layer~\cite{SaitoTadanoImakura2016TJSIAM,YanoFutamuraImakuraSakurai2018}.
\end{itemize}

This work considers an extension of such a projection method to singular cases.
To this end, motivated by a consideration of spectra of a single nonsquare matrix~\cite[Section 6.7]{Ben-IsraelGreville2003}, the resolvent matrix in the projection method is replaced with the pseudoinverse of the linear matrix pencil.
The integral is also approximated by a numerical quadrature, similarly to the regular case.
Each quadrature point gives a (block) least squares problem to solve, and it can be solved independently.
This extension carries out those features of the above projection methods in parallel by nature.
Note that the proposed method does not attempt to determine the sizes of Kronecker blocks, and thus, it is not naively comparable to GUPTRI.
Also, GUPTRI is not designed to compute the eigenvalues in a prescribed region and the corresponding eigenvectors and is not designed to be efficient for sparse matrices $A$ and $B$.
Note also that the proposed method will be considered in theory for linear matrix pencils whose singular blocks are of size zero, as will be formally assumed later.

\subsection{Preliminary} \label{sec:preliminary}
The spectral properties of the linear matrix pencil $z B - A$ are given by the Kronecker canonical form~\cite[Chapter XII]{Gantmacher1959}, \cite[Section~2.1]{KunkelMehrmann2006}, \cite[Theorem~7.8.1]{Bernstein2018}
\begin{equation}
	P (zB - A) Q = G (z) \oplus K (z) \oplus L (z) \oplus U (z) \in \mathbb{C}^{m \times n},
	\label{eq:KCF}
\end{equation}
where $P \in \mathbb{C}^{m \times m}$ and $Q \in \mathbb{C}^{n \times n}$ are nonsingular matrices,$\oplus$ denotes the direct sum of matrices
\begin{align}
	G (z) & = \oplus_{i=1}^\ell G_i (z) \in \mathbb{C}^{\eta \times \eta}, \\\
	G_i (z) & = z \mathrm{I}_{\eta_i} - J_i \in \mathbb{C}^{\eta_i \times \eta_i}, \quad i = 1, 2, \dots, \ell, \\
	K (z) & = \oplus_{i=1}^p K_i (z) \in \mathbb{C}^{\rho \times \rho}, \\
	K_i (z) & = z N_{\rho_i} - \mathrm{I}_{\rho_i} \in \mathbb{C}^{\rho_i \times \rho_i}, \quad i = 1, 2, \dots, p, \\
	L (z) & = \oplus_{i=1}^q L_i (z) \in \mathbb{C}^{\mu \times (\mu + q)}, \\
	L_i (z) & = z \left[\mathrm{I}_{\mu_i}, \boldsymbol{0}\right]- \left[N_{\mu_i}, \boldsymbol{e}_{\mu_i}\right] \in \mathbb{C}^{\mu_i \times (\mu_i + 1)}, \quad i = 1, 2, \dots, q, \\
	U (z) & = \oplus_{i=1}^r U_i (z) \in \mathbb{C}^{(\nu + r) \times \nu}, \\
	U_i (z) & = z \left[\mathrm{I}_{\nu_i}, \boldsymbol{0}\right]^\mathsf{T} - \left[N_{\nu_i}, \boldsymbol{e}_{\nu_i}\right]^\mathsf{T} \in \mathbb{C}^{(\nu_i + 1) \times \nu_i}, \quad i = 1, 2, \dots, r,
\end{align}
$\eta = \sum_{i=1}^\ell \eta_i$, $\rho = \sum_{i=1}^p \rho_i$, $\mu = \sum_{i=1}^q \mu_i$, $\nu = \sum_{i=1}^r \nu_i$, 
\begin{equation}
	J_i = 
	\begin{bmatrix}
		\lambda_i & 1 & & & \bigzerou \\
		& \lambda_i & 1 \\
		& & \ddots & \ddots \\
		& & & & 1 \\
		\bigzerol & & & & \lambda_i
	\end{bmatrix} \in \mathbb{C}^{\eta_i \times \eta_i}
\end{equation}
is a Jordan block corresponding to finite eigenvalues $\lambda_i$, and $N_{\rho_i} \in \mathbb{C}^{\rho_i \times \rho_i}$ is the shift matrix whose superdiagonal entries are $1$ and the remaining entries are zero.
Here, we denote the identity matrix of size $\eta$ by $\mathrm{I}_\eta \in \mathbb{R}^{\eta \times \eta}$ and the $i$th standard basis vector by $\boldsymbol{e}_i$.
Note that~$G (z)$ and $K (z)$ have the finite and infinite eigenvalues of $zB - A$, respectively.
We call the block matrices $G_i (z)$ and $K_i (z)$ regular blocks.
Hence, matrix pencil~$zB - A$ has $\eta$ finite eigenvalues~$\lambda_{\eta_i}$, $i = 1$, $2$, $\dots$, $\ell$, with multiplicity~$\eta_i$, respectively, and $\rho$ infinite eigenvalues.
Meanwhile, we call $L_i (z)$ and $U_i (z)$ the right and left singular blocks, respectively.
Hence, $q$ and $r$ are the numbers of right and left singular blocks, respectively.
Let $t$ be the number of finite eigenvalues of matrix pencil~$z B - A$ counting multiplicity in the prescribed region~$\Omega$.
Then, without loss of generality, the eigenvalues in the region~$\Omega$ are denoted by $\lambda_1$, $\lambda_2$, $\dots$, $\lambda_s$, and the associated regular blocks are denoted by $G_1 (z)$, $G_2 (z)$, $\dots$, $G_s (z)$.
Then, let $J_\Omega = \oplus_{i=1}^s J_i \in \mathbb{C}^{t \times t}$.
Hence, $t = \sum_{i=1}^s \eta_i$ holds.

Throughout, matrix pencil $zB-A$ is assumed to have left and right singular blocks of size $\mu = \nu = 0$ if not specified otherwise.
Hence, the proposed method aims at determining the eigenvalues of $J_i$ for all $i$ such that $\lambda_i \in \Omega$.
Note that the proposed method does not aim to determine the Kronecker structures such as the minimal indices~$\mu_i$ and $\nu_i$.
Further, the matrix pencil $z B - A$ is assumed to be exact.
In different contexts, perturbations in matrices $A$ and $B$ are assumed \cite{BoutryEladGolubMilanfar2005, DasNeumaier2013, ItoMurota2016}.

\subsection{Organization}
\Cref{sec:proposed} gives the formulation of the proposed method and the condition such that the method determines the desired eigenpairs, and presents the implementation issues.
\Cref{sec:nuex} presents the results of numerical experiments on test matrix pencils.
Finally, \Cref{sec:conc} concludes the paper.

\section{Proposed method} \label{sec:proposed}
This work considers extending the scope of problems that can be solved by a projection method for regular matrix pencils (CIRR~\cite{SakuraiTadano2007, IkegamiSakurai2010}) to singular nonsquare matrix pencils.
The extended method reduces a given nonsquare matrix eigenproblem \eqref{eq:evp} to a generalized square eigenproblem.
The main principle of this approach lies in the $k$th order complex moment matrix, given by the contour integral
\begin{align}
	\mathsf{M}_k = \frac{1}{2 \pi \mathrm{i}} \oint_\Gamma z^k (zB - A)^\dag \mathrm{d}z \in \mathbb{C}^{n \times m}, \quad k = 0, 1, \dots, M - 1,
	\label{eq:moment}
\end{align}
where $\pi$ is the circular constant, $\mathrm{i}$ is the imaginary unit, $\Gamma$ is a positively-oriented closed Jordan curve of which $\Omega$ is the interior, $^\dag$ denotes the Moore--Penrose generalized inverse (pseudoinverse) of a matrix, and $M$ is the order of moments such as $M < n$. (See~\cite[Section~6.7]{Ben-IsraelGreville2003} for the spectral properties of nonsquare matrices.)
Throughout, it is assumed that no eigenvalues of $J_i$ exist in $\Gamma$ for all~$i = 1$, $2$, $\dots$, $\ell$.

We reduce \eqref{eq:evp} to a smaller eigenproblem having the eigenvalues~$\lambda \in \Omega$ of interest.
Let $L \in \mathbb{Z}_{>0}$ such as $L < \min(m, n)$, $T \in\mathbb{C}^{LM \times m}$ and $V \in \mathbb{C}^{m \times L}$ be random matrices, and
\begin{align}
	S & = \left[S_0, S_1, \dots, S_{M-1}\right] \in \mathbb{C}^{n \times LM}, \quad S_k = \mathsf{M}_k V \in \mathbb{C}^{n \times L}. 
	\label{eq:def_S}
\end{align}
Then, the reduction of \eqref{eq:evp} to the square generalized eigenproblem
\begin{align}
	T \! A S \boldsymbol{y} = \lambda T \! B S \boldsymbol{y}, \quad \boldsymbol{y} \in \mathbb{C}^{L M} \setminus \lbrace \boldsymbol{0} \rbrace, \quad \lambda \in \mathbb{C}
	\label{eq:TAS}
\end{align}
is yielded by a Rayleigh--Ritz-like procedure \cite[\S4.3]{Saad2011} over the subspace $\mathcal{R}(S)$ to obtain the desired eigenvalues, where $\mathcal{R}(\cdot)$ denotes the range of a matrix.
\bigskip

We deal with both cases $m \geq n$ and $m < n$ in a unified manner.
To derive conditions such that \eqref{eq:TAS} has eigenpairs of interest, we prepare a form of matrix pencil~$zB-A$.
From the Kronecker canonical form~\eqref{eq:KCF}, the pseudoinverse of matrix pencil~$z B - A$ has the form
\begin{align}	
	& (z B - A)^\dag \notag \\
	& \quad = \left[ P^{-1} (G(z) \oplus K(z) \oplus L(z) \oplus U(z)) Q^{-1} \right]^\dag \notag \\
	& \quad = \left\lbrace \left[ P^{-1} \left( \mathrm{I}_{\eta+\rho+\mu} \oplus U(z) \right) \right] \left[ \left( G(z) \oplus K(z) \oplus L(z) \oplus \mathrm{I}_\nu \right) Q^{-1} \right] \right\rbrace^\dag \notag \\	
	& \quad = \left[ \left( G(z) \oplus K(z) \oplus L(z) \oplus \mathrm{I}_\nu \right) Q^{-1} \right]^\dag \left[ P^{-1} \left( \mathrm{I}_{\eta+\rho+\mu} \oplus U(z) \right) \right]^\dag \notag \\
	& \quad = \left\lbrace \left( G(z) \oplus K(z) \oplus \mathrm{I}_{\mu+\nu} \right) \left[ \mathrm{I}_{\eta+\rho} \oplus L(z) \oplus \mathrm{I}_{\nu} \right] Q^{-1} \right\rbrace^\dag \left[ P^{-1} \left( \mathrm{I}_{\eta+\rho+\nu} \oplus U(z) \right) \right]^\dag \notag \\	
	& \quad = \left[ \left( \mathrm{I}_{\eta+\rho} \oplus L(z) \oplus \mathrm{I}_{\nu} \right) Q^{-1} \right]^\dag \left( G(z) \oplus K(z) \oplus \mathrm{I}_{\mu+\nu} \right)^{-1} \left[ P^{-1} \left( \mathrm{I}_{\eta+\rho+\nu} \oplus U(z) \right) \right]^\dag.	
	\label{eq:(zB-A)^dag}	
\end{align}
Here, we used the fact~\cite[Fact~8.4.23]{Bernstein2018} that if $E \in \mathbb{C}^{\ell \times m}$ has full-column rank and $F \in \mathbb{C}^{m \times n}$ has full-row rank, then $E^\dag E = F F^\dag = \mathrm{I}_m$ and	
\begin{align}	
	(E F)^\dag & = (E^\dag E F)^\dag (E F F^\dag)^\dag \\	
	& = F^\dag E^\dag.	
\end{align}
Hence, $(zB-A)^\dag$ is decomposed into three full-rank matrices~\eqref{eq:(zB-A)^dag}: the first is related to the right singular blocks, the second is related to the regular blocks, and the third is related to the left singular blocks.

Assume that the singular Kronecker blocks $L(z)$, $U (z)$ have size~$\mu = \nu = 0$.
Then, the decomposition~\eqref{eq:(zB-A)^dag} is reduced into
\begin{align}	
	(z B - A)^\dag = \left( \left[ \mathrm{I}_{\eta+\rho}, \mathrm{O}_{\eta+\rho, q} \right] Q^{-1} \right)^\dag \left( G(z) \oplus K(z) \right)^{-1} \left( P^{-1} 
	\begin{bmatrix}
		\mathrm{I}_{\eta+\rho} \\
		\mathrm{O}_{r, \eta+\rho}
	\end{bmatrix}	
	\right)^\dag,
\end{align}
where $\mathrm{O}_{r, q}$ is the zero matrix of size~$r \times q$.
The next lemma gives useful expressions of the first factor~$( [ \mathrm{I}_{\eta+\rho}, \mathrm{O} ] Q^{-1} )^\dag$ and third factor~$( P^{-1} [ \begin{smallmatrix} \mathrm{I}_{\eta+\rho} \\ \mathrm{O} \end{smallmatrix} ] )^\dag$.
Here, $\Pi_\mathcal{S}$ denotes the orthogonal projector onto a subspace~$\mathcal{S}$ and $\mathcal{S}^\perp$ denotes the orthogonal complement of $\mathcal{S}$.

\begin{lemma} \label{lm:C}
	Let $C \in \mathbb{C}^{n \times n}$ be a nonsingular matrix.
	Partition 
	\begin{align}
		C = \left[ C_1, C_2 \right]
	\end{align}
	columnwise into submatrices $C_1 \in \mathbb{C}^{n \times k}$ and $C_2 \in \mathbb{C}^{n \times (n-k)}$.
	Then, we have
	\begin{align}
		\left( \left[ \mathrm{I}_k, \mathrm{O}_{k, n-k} \right] C^{-1} \right)^\dag = \Pi_{\mathcal{R}(C_2)^\perp} C_1.
	\end{align}
\end{lemma}

\begin{proof}
	Because $\left[ \mathrm{I}_k, \mathrm{O} \right] C^{-1}$ has full-row rank, we have 
	\begin{align}
		\left( \left[ \mathrm{I}_k, \mathrm{O} \right] C^{-1} \right)^\dag
		& = \left( \left[ \mathrm{I}_k, \mathrm{O} \right] C^{-1} \right)^\mathsf{H} \left[ \left( \left[ \mathrm{I}_k, \mathrm{O} \right] C^{-1} \right) \left( \left[ \mathrm{I}_k, \mathrm{O} \right] C^{-1} \right)^\mathsf{H} \right]^{-1} \\
		& = C \left( C^\mathsf{H} C \right)^{-1}
		\begin{bmatrix}
			\mathrm{I}_k \\
			\mathrm{O}
		\end{bmatrix}
		\left\lbrace
		\left[ \mathrm{I}_k, \mathrm{O} \right] (C^\mathsf{H} C)^{-1}
		\begin{bmatrix}
			\mathrm{I}_k \\
			\mathrm{O}
		\end{bmatrix}
		\right\rbrace^{-1}. 
	\end{align}
	Noting that the Schur complement of $C_2^\mathsf{H} C_2$ in the $2 \times 2$ block matrix	
	\begin{align}
		C^\mathsf{H} C = 
		\begin{bmatrix}
			{C_1}^\mathsf{H} C_1 & {C_1}^\mathsf{H} C_2 \\
			{C_2}^\mathsf{H} C_1 & {C_2}^\mathsf{H} C_2 
		\end{bmatrix}
	\end{align}
	is $\mathsf{S} = {C_1}^\mathsf{H} C_1 - {C_1}^\mathsf{H} C_2 ({C_2}^\mathsf{H} C_2)^{-1} {C_2}^\mathsf{H} C_1$ \cite[section~0.7.3]{HornJohnson2013},
	we obtain
	\begin{align}
		\left( \left[ \mathrm{I}_k, \mathrm{O} \right] C^{-1} \right)^\dag 
		& = \left[ C_1, C_2 \right]
		\begin{bmatrix}
			\mathsf{S}^{-1} \\
			- {C_2}^\dag C_1 \mathsf{S}^{-1}
		\end{bmatrix}
		\mathsf{S} \\
		& = C_1 - C_2 {C_2}^\dag C_1 \\
		& = \Pi_{\mathcal{R}(C_2)^\perp} C_1.
	\end{align}
\end{proof}

Partition $Q = \left[ Q_\mathrm{r}, Q_\mathrm{s} \right]$ columnwise into submatrices~$Q_\mathrm{r} \in \mathbb{C}^{n \times (\eta+\rho)}$ and \break $Q_\mathrm{s} \in \mathbb{C}^{n \times r}$ corresponding to the regular and right singular Kronecker blocks, respectively.
Applying \cref{lm:C} to $Q$, we obtain
\begin{align}
	\left( \left[ \mathrm{I}_{\eta+\rho}, \mathrm{O} \right] Q^{-1} \right)^\dag = \Pi_{\mathcal{R}(Q_\mathrm{s})^\perp} Q_\mathrm{r}.
\end{align}
Additionally, partition 
\begin{align}
	P = 
	\begin{bmatrix}
		P_\mathrm{r} \\
		P_\mathrm{s}
	\end{bmatrix}
\end{align}
rowwise into submatrices $P_\mathrm{r} \in \mathbb{C}^{(\eta+\rho) \times m}$ corresponding to the regular Kronecker blocks and $P_\mathrm{s} \in \mathbb{C}^{r \times m}$ corresponding to the left singular Kronecker blocks.
Transposing the statements in \cref{lm:C}, we obtain 
\begin{align}
	\left(
	P^{-1} 
	\begin{bmatrix}
		\mathrm{I}_{\eta+\rho} \\
		\mathrm{O}
	\end{bmatrix}
	\right)^\dag
	= P_\mathrm{r} \Pi_{\mathcal{R}({P_\mathrm{s}}^\mathsf{H})^\perp}.
\end{align}

For convenience, as per the Kronecker block structure, we partition 
\begin{align}
	P = 
	\begin{bmatrix}
		P_1 \\
		P_2 \\
		\vdots \\
		P_{\ell+p+r}
	\end{bmatrix}, \quad
	P_i \in
	\begin{cases}
		\mathbb{C}^{\eta_i \times m}, & i=1, 2, \dots, \ell, \\
		\mathbb{C}^{\rho_i \times m}, & i=\ell+1, \ell+2, \dots, \ell+p, \\
		\mathbb{C}^{1 \times m}, & i=\ell+p+1, \ell+p+2, \dots, \ell+p+r,
	\end{cases}
\end{align}
rowwise into submatrices corresponding to the regular Kronecker blocks regarding the finite eigenvalues, the regular Kronecker blocks regarding the infinite eigenvalues, and the left singular blocks, respectively.
Here, the first column of $P_i$ is the left eigenvector corresponding to $\lambda_i$ and the remaining columns of $P_i$ are the generalized right eigenvectors for $i=1$, $2$, $\dots$, $\ell$.
And, partition
\begin{align}
	\left( P^{-1} 
	\begin{bmatrix}
		\mathrm{I}_{\eta+\rho} \\
		\mathrm{O}
	\end{bmatrix}
	\right)^\dag
	= 
	\begin{bmatrix}
		\hat{P}_1 \\
		\hat{P}_2 \\
		\vdots \\
		\hat{P}_{\ell+p}
	\end{bmatrix} \in \mathbb{C}^{(\eta+\rho) \times m}
\end{align}
rowwise into submatrices 
\begin{align}
	\hat{P}_i & = P_i \Pi_{\mathcal{R}({P_\mathrm{s}}^\mathsf{H})^\perp} \in \mathbb{C}^{\eta_i \times m}, \quad i=1, 2, \dots, \ell, \label{eq:Phat_i} \\
	\hat{P}_{\ell+i} & = P_{\ell+i} \Pi_{\mathcal{R}({P_\mathrm{s}}^\mathsf{H})^\perp} \in \mathbb{C}^{\rho_i \times m}, \quad i=1, 2, \dots, p. \label{eq:Phat_l+i}
\end{align}
Furthermore, according to the Kronecker block structure, we partition 
\begin{align}
	Q = \left[ Q_1, Q_2, \dots, Q_{\ell+p+q} \right], \quad
	Q_i \in 
	\begin{cases}
		\mathbb{C}^{n \times \eta_i}, & i = 1, 2, \dots, \ell, \\
		\mathbb{C}^{n \times \rho_i}, & i = \ell+1, \ell+2, \dots, \ell+p, \\
		\mathbb{C}^{n \times 1}, & i = \ell+p+1, \ell+p+2, \dots, \ell+p+q
	\end{cases}
\end{align}
columnwise into submatrices corresponding to the regular Kronecker blocks regarding the finite eigenvalues, the regular Kronecker blocks regarding the infinite eigenvalues, and the right singular blocks, respectively.
Here, the first column of $Q_i$ is the right eigenvector corresponding to $\lambda_i$ and the remaining columns of $Q_i$ are the generalized right eigenvectors for $i=1$, $2$, $\dots$, $\ell$.
Further, partition 
\begin{align}
	\left( \left[ \mathrm{I}_{\eta+\rho}, \mathrm{O} \right] Q^{-1} \right)^\dag = \left[\hat{Q}_1, \hat{Q}_2, \dots, \hat{Q}_{\ell+p} \right] \in \mathbb{C}^{n \times (\eta+\rho)}
\end{align}
columnwise into 
\begin{align}
	\hat{Q}_i & = \Pi_{\mathcal{R}(Q_\mathrm{s})^\perp} Q_i \in \mathbb{C}^{n \times \eta_i}, \quad i = 1, 2, \dots, \ell, \label{eq:Qhat_i} \\
	\hat{Q}_{\ell+i} & = \Pi_{\mathcal{R}(Q_\mathrm{s})^\perp} Q_{\ell+i} \in \mathbb{C}^{n \times \rho_i}, \quad i = 1, 2, \dots, p. \label{eq:Qhat_l+i}
\end{align}

Now, we express the moment matrix~$\mathsf{M}_k$ of order~$k$ using the Kronecker canonical form via the expansion
\begin{align}
	(z B - A)^\dag = \sum_{i=1}^\ell \hat{Q}_i G_i (z)^{-1} \hat{P}_i + \sum_{i=1}^{p} \hat{Q}_{\ell+i} K_i (z)^{-1} \hat{P}_{\ell+i}.
	\label{eq:(zB-A)^dag_expansion}
\end{align}
The next lemma gives the localized moment matrix in a simple form.
\begin{lemma} \label{lm:Mk}
	Let
	\begin{align}
		\hat{P}_\Omega & =
		\begin{bmatrix}
			\hat{P}_1 \\
			\hat{P}_2 \\
			\vdots \\
			\hat{P}_s \\
		\end{bmatrix}
		\in \mathbb{C}^{t \times m}, \label{eq:def_Phat_Omega} \\
		\hat{Q}_\Omega & = \left[ \hat{Q}_1, \hat{Q}_2, \dots, \hat{Q}_s \right] \in \mathbb{C}^{n \times t}, \label{eq:def_Qhat_Omega}
	\end{align}
	where $\hat{P}_i$ and $\hat{Q}_i$ are defined by \eqref{eq:Phat_i}--\eqref{eq:Phat_l+i} and \eqref{eq:Qhat_i}--\eqref{eq:Qhat_l+i}, respectively.
	If $\mathsf{M}_k$ is the moment matrix \eqref{eq:moment}, then we have
	\begin{align}
		\mathsf{M}_k & = \sum_{i=1}^s \hat{Q}_i J_i^k \hat{P}_i \\
		& = \hat{Q}_\Omega J_\Omega^k \hat{P}_\Omega.
		\label{eq:decomposition_Mk}
	\end{align}
\end{lemma}

\begin{proof}
	The expansion~\eqref{eq:(zB-A)^dag_expansion} gives
	\begin{align}
		\mathsf{M}_k & = \sum_{i=1}^\ell \hat{Q}_i \left[ \frac{1}{2 \pi \mathrm{i}} \int_\Gamma z^k G_i (z)^{-1} \mathrm{d} z \right] \hat{P}_i +  \sum_{j=1}^p \hat{Q}_{\ell+i} \left[ \frac{1}{2 \pi \mathrm{i}} \int_\Gamma z^k K_i(z)^{-1} \mathrm{d} z \right] \hat{P}_{\ell+i},
	\end{align}
	in which two contour integrals of $G_i (z)$ and $K_i (z)$ are considered.
	The inverse of a regular Kronecker block corresponding to a finite eigenvalue is of the form
	\begin{align}
		G_i (z)^{-1} = \sum_{j=0}^{\eta_i-1} \frac{1}{(z-\lambda_i)^{j+1}} (J_i - \lambda_i \mathrm{I}_{\eta_i})^j.
	\end{align}
	From the assumption, $G_i (z)$ is regular for all $z \in \Gamma$.
	By using Cauchy's integral formula,	we have
	\begin{align}
		& \frac{1}{2 \pi \mathrm{i}} \oint_\Gamma z^k G_i (z)^{-1} \mathrm{d}z \\
		& = \sum_{j=0}^{\min(k, \eta_i-1)} \left[\frac{1}{2 \pi \mathrm{i}} \oint_\Gamma \frac{z^k}{(z-\lambda_i)^{j+1}} \mathrm{d} z \right] (J_i - \lambda_i \mathrm{I}_{\eta_i})^j \\
		& = 
		\begin{cases}	
			\sum_{j=0}^{\min(k, \eta_i-1)} \binom{k}{k-j} {\lambda_i}^{k-j} (J_i - \lambda_i \mathrm{I}_{\eta_i})^j \quad & \text{for} ~ i ~ \text{such that} ~ \lambda_i \in \Omega, \\
			\mathrm{O} \quad & \text{for} ~ i ~ \text{such that} ~ \lambda_i \not \in \Omega
		\end{cases} \\
		& = 
		\begin{cases}
			{J_i}^k \quad & \text{for} ~ i ~ \text{such that} ~ \lambda_i \in \Omega, \\
			\mathrm{O} \quad & \text{for} ~ i ~ \text{such that} ~ \lambda_i \not \in \Omega,
		\end{cases}
		\label{eq:GJ}
	\end{align}
	where $\binom{k}{i}$ is the binomial coefficient.
	Because $K_i (z)^{-1}$ is regular for all $z \in \mathbb{C}$, we have
	\begin{align}
		\frac{1}{2 \pi \mathrm{i}} \oint_\Gamma z^k {K_i (z)}^{-1} \mathrm{d} z = \mathrm{O}.
		\label{eq:K0}
	\end{align}	
	Therefore, we obtain \eqref{eq:decomposition_Mk}.
\end{proof}

The decomposition~\eqref{eq:decomposition_Mk} reduces to that in \cite[equation~(6)]{IkegamiSakurai2010}, when $m = n$ and no singular Kronecker blocks $L(z)$ and $U(z)$ exist.
Hence, the decomposition~\eqref{eq:decomposition_Mk} is a generalization to the singular case.
Also, \eqref{eq:decomposition_Mk} extends the formula~\cite[equation~(5.8)]{GuttelTisseur2017} to singular linear matrix pencils.

In the following lemma, the localized moment matrix~\eqref{eq:decomposition_Mk} is used to extract the Jordan blocks corresponding to the eigenvalues located in the region~$\Omega$.

\begin{lemma}[{cf.~\cite[Theorem~5]{IkegamiSakuraiNagashima2010}}] \label{lm:Th5_Ikegami_etal}
	Let $L$, $M \in \mathbb{Z}_{> 0}$ such that $\min(m, n) > L \geq t$, $W \in \mathbb{C}^{L \times n}$, $V \in \mathbb{C}^{m \times L}$ be arbitrary matrices, $t$ be the number of eigenvalues of matrix pencil~$z B - A$ in a region~$\Omega$, and $\hat{P}_\Omega$, $\hat{Q}_\Omega$ be defined by \eqref{eq:def_Phat_Omega}, \eqref{eq:def_Qhat_Omega}, respectively.
	If $\mathrm{rank} (W \hat{Q}_\Omega) = \mathrm{rank} (\hat{P}_\Omega V) = t$, then the nonsingular part of the size-reduced matrix pencil~$z W \mathsf{M}_0 V - W \mathsf{M}_1 V$ is equivalent to $z \mathrm{I} - J_\Omega$.
\end{lemma}
\begin{proof}
	Because $\mathrm{rank} (W \hat{Q}_\Omega) = \mathrm{rank} (\hat{P}_\Omega V) = t \leq L$, there exist nonsingular matrices $\mathsf{P} \in \mathbb{C}^{L \times L}$ and $\mathsf{Q} \in \mathbb{C}^{L \times L}$ such that
	\begin{align}
		\hat{P}_\Omega V \mathsf{P} = \left[ \mathrm{I}_t, \mathrm{O} \right], \quad 
		\mathsf{Q} W \hat{Q}_\Omega = 
		\begin{bmatrix}
			\mathrm{I}_t \\
			\mathrm{O}
		\end{bmatrix}.
	\end{align}
	Hence, we have
	\begin{align}
		\mathsf{Q} (z W \mathsf{M}_0 V - W \mathsf{M}_1 V) \mathsf{P} & = \mathsf{Q} \left( z W \hat{Q}_\Omega \mathrm{I}_t \hat{P}_\Omega V - W \hat{Q}_\Omega J_\Omega \hat{P}_\Omega V \right) \mathsf{P} \\
		& = (z \mathrm{I}_t - J_\Omega) \oplus \mathrm{O}.
	\end{align}
\end{proof}

\Cref{lm:Th5_Ikegami_etal} shows that solving the size-reduced eigenproblem~$W \mathsf{M}_1 V \boldsymbol{y} = \break \lambda W \mathsf{M}_0 V \boldsymbol{y}$, $\boldsymbol{y} \ne \boldsymbol{0}$ gives the desired eigenvalues.
Based on \cref{lm:Th5_Ikegami_etal}, the following theorem demonstrates an approach to project the matrix pencil $z B - A$ onto the desired Ritz space to form a size-reduced matrix pencil having the desired eigenvalues.

\begin{theorem} \label{th:equivalence}
	Let $S$ be defined in~\eqref{eq:def_S} and $T \in \mathbb{C}^{LM \times m}$ and $V \in \mathbb{C}^{m \times L}$ be arbitrary matrices.
	If $\mathrm{rank} (\hat{P}_\Omega V) = \mathrm{rank} \left( T P^{-1} [\mathrm{I}_t, \mathrm{O}]^\mathsf{T} \right) = t$, then the nonsingular part of the reduced moment matrix pencil~$zTBS - TAS$ is equivalent to $z \mathrm{I}_t - J_\Omega$.
\end{theorem}

\begin{proof}
	The coefficient matrices are expressed as 
	\begin{align}
		A & = P^{-1} 
		\left[ (\oplus_{i=1}^\ell J_i) \oplus \mathrm{I}_\rho \oplus \mathrm{O}_{r \times q} \right]
		Q^{-1} \\
		& = P^{-1} 
		\begin{bmatrix}
			\mathrm{I}_{\eta+\rho} \\
			\mathrm{O}
		\end{bmatrix}	
		\left[ (\oplus_{i=1}^\ell J_i) \oplus \mathrm{I}_\rho \right] \left[ \mathrm{I}_{\eta+\rho}, \mathrm{O} \right]
		Q^{-1}, \\
		B & = P^{-1} \left[ \mathrm{I}_\eta \oplus \left( \oplus_{i=1}^p N_{\rho_i} \right) \oplus \mathrm{O}_{r \times q} \right] Q^{-1} \\
		& = P^{-1} 
		\begin{bmatrix}
			\mathrm{I}_{\eta+\rho} \\
			\mathrm{O}
		\end{bmatrix}	
		\left[ \mathrm{I}_\eta \oplus \left( \oplus_{i=1}^p N_{\rho_i} \right) \right] \left[ \mathrm{I}_{\eta+\rho}, \mathrm{O} \right]
		Q^{-1}.
	\end{align}
	Then, we have
	\begin{align}
		P A \mathsf{M}_k & =
		\begin{bmatrix}
			\mathrm{I}_{\eta+\rho} \\
			\mathrm{O}
		\end{bmatrix}	
		\left[ \left( \oplus_{i=1}^\ell J_i \right) \oplus \mathrm{I}_\rho \right] \left[ \mathrm{I}_{\eta+\rho}, \mathrm{O} \right] \left[ Q^{-1} \left(\mathrm{I}_n - Q_\mathrm{s} {Q_\mathrm{s}}^\dag \right) Q_\Omega \right] {J_\Omega}^k \hat{P}_\Omega \\
		& =
		\begin{bmatrix}
			\mathrm{I}_{\eta+\rho} \\
			\mathrm{O}
		\end{bmatrix} \left( \left[ \left( \oplus_{i=1}^\ell J_i \right) \oplus \mathrm{I}_\rho \right] \left[ \mathrm{I}_{\eta+\rho}, \mathrm{O} \right]
		\begin{bmatrix}
			\mathrm{I}_t \\
			\mathrm{O} \\
			- {Q_\mathrm{s}}^\dag Q_\Omega
		\end{bmatrix} \right) {J_\Omega}^k \hat{P}_\Omega \\
		& = 
		\begin{bmatrix} 
			\mathrm{I}_{\eta+\rho} \\
			\mathrm{O}
		\end{bmatrix}
		\begin{bmatrix}
			J_\Omega \\
			\mathrm{O}
		\end{bmatrix} 
		{J_\Omega}^k \hat{P}_\Omega \\
		& = 
		\begin{bmatrix} 
			\mathrm{I}_{\eta+\rho} \\
			\mathrm{O}
		\end{bmatrix}
		\begin{bmatrix}
			\mathrm{I}_t \\
			\mathrm{O}
		\end{bmatrix}
		{J_\Omega}^{k+1} \hat{P}_\Omega \\
		& = \left( \mathrm{I}_{\eta+\rho} \oplus \mathrm{O} \right) Q^{-1} \left( \Pi_{\mathcal{R}(Q_\mathrm{s})^\perp} Q_\Omega {J_\Omega}^{k+1} \hat{P}_\Omega \right) \\
		& = \left( \mathrm{I}_{\eta+\rho} \oplus \mathrm{O} \right) Q^{-1} \mathsf{M}_{k+1}
	\end{align}
	and
	\begin{align}
		P B \mathsf{M}_k & = 
		\begin{bmatrix}
			\mathrm{I}_{\eta+\rho} \\
			\mathrm{O}
		\end{bmatrix}	
		\left[ \mathrm{I}_\eta \oplus \left( \oplus_{i=1}^p N_{\rho_i} \right) \right] \left[ \mathrm{I}_{\eta+\rho}, \mathrm{O} \right]
		\left[ Q^{-1} \left( \mathrm{I}_n - Q_\mathrm{s} {Q_\mathrm{s}}^\dag \right) Q_\Omega \right] {J_\Omega}^k \hat{P}_\Omega \\
		& = 
		\begin{bmatrix}
			\mathrm{I}_{\eta+\rho} \\
			\mathrm{O}
		\end{bmatrix}	
		\left( \left[ \mathrm{I}_\eta \oplus \left( \oplus_{i=1}^p N_{\rho_i} \right) \right]
		\left[ \mathrm{I}_{\eta+\rho}, \mathrm{O} \right]
		\begin{bmatrix}
			\mathrm{I}_t \\
			\mathrm{O} \\
			- {Q_\mathrm{s}}^\dag Q_\Omega
		\end{bmatrix} \right)
		{J_\Omega}^k \hat{P}_\Omega \\
		& = 
		\begin{bmatrix}
			\mathrm{I}_{\eta+\rho} \\
			\mathrm{O}
		\end{bmatrix}	
		\begin{bmatrix}
			\mathrm{I}_t \\
			\mathrm{O}
		\end{bmatrix}
		{J_\Omega}^k \hat{P}_\Omega \\
		& = \left( \mathrm{I}_{\eta+\rho} \oplus \mathrm{O} \right) Q^{-1} \left( \Pi_{\mathcal{R}(Q_\mathrm{s})^\perp} Q_\Omega {J_\Omega}^k \hat{P}_\Omega \right) \\
		& = \left( \mathrm{I}_{\eta+\rho} \oplus \mathrm{O} \right) Q^{-1} \mathsf{M}_k.
	\end{align}
	Here, we used 
	\begin{align}
		Q^{-1} Q_\Omega & = 
		\begin{bmatrix}
			\mathrm{I}_t \\
			\mathrm{O}
		\end{bmatrix} \in \mathbb{R}^{n \times t}, \\
		Q^{-1} Q_\mathrm{s} & = 
		\begin{bmatrix}		
			\mathrm{O} \\
			\mathrm{I}_q 
		\end{bmatrix} \in \mathbb{R}^{n \times q}, \\
		\left( \left[ \mathrm{I}_{\eta+\rho}, \mathrm{O} \right] Q^{-1} \right) \left( \left[ \mathrm{I}_{\eta+\rho}, \mathrm{O} \right] Q^{-1} \right)^\dag & = \left( \left[ \mathrm{I}_{\eta+\rho}, \mathrm{O} \right] Q^{-1} \right) \Pi_{\mathcal{R}(Q_\mathrm{s})^\perp} \left[ Q_\Omega, Q_{s+1}, \dots Q_{\ell+p} \right] \\
		& = \mathrm{I}_t \oplus \mathrm{I}_{\eta+\rho-t}.
	\end{align}
	Now, let
	\begin{align}
		W = T P^{-1} (\mathrm{I}_{\eta+\rho} \oplus \mathrm{O}) Q^{-1}.
		\label{eq:def_C}
	\end{align}
	Then, we have
	\begin{align}
		T A \mathsf{M}_k V & = W \mathsf{M}_{k+1} V, \\
		T B \mathsf{M}_k V & = W \mathsf{M}_k V.
	\end{align}
	Applying $\left( \left[ \mathrm{I}_{\eta+\rho}, \mathrm{O} \right] Q^{-1} \right)^\dag$ to both sides of the equation~\eqref{eq:def_C}, we have
	\begin{align}
		W \left( \left[ \mathrm{I}_{\eta+\rho}, \mathrm{O} \right] Q^{-1} \right)^\dag = T P^{-1} 
		\begin{bmatrix}
			\left[ \mathrm{I}_{\eta+\rho}, \mathrm{O} \right] Q^{-1} \left( \left[ \mathrm{I}_{\eta+\rho}, \mathrm{O} \right] Q^{-1} \right)^\dag \\
			\mathrm{O}
		\end{bmatrix},
	\end{align}
	or 
	\begin{align}
		W \Pi_{\mathcal{R}(Q_\mathrm{s})^\perp} Q_\mathrm{r} = T P^{-1} 
		\begin{bmatrix}
			\mathrm{I}_{\eta+\rho} \\
			\mathrm{O}
		\end{bmatrix}.
		\label{eq:CQ}
	\end{align}
	The first~$t$ columns of both sides of \eqref{eq:CQ} form
	\begin{align}
		W \hat{Q}_\Omega = T P^{-1} 
		\begin{bmatrix}
			\mathrm{I}_t \\
			\mathrm{O}
		\end{bmatrix}.
	\end{align}
	Hence, from the assumption, we have
	\begin{align}
		\mathrm{rank} (W \hat{Q}_\Omega) & = \mathrm{rank} \left( T P^{-1} 
		\begin{bmatrix}
			\mathrm{I}_t \\
			\mathrm{O}
		\end{bmatrix} 
		\right) \\
		& = t.
	\end{align}
	Therefore, because $\mathrm{rank}(W \hat{Q}_\Omega) = \mathrm{rank} (\hat{P}_\Omega V) = t$, \cref{lm:Th5_Ikegami_etal} gives the assertion.	
\end{proof}

The following lemma prepares the proof of the main theorem.

\begin{lemma} \label{lm:over}
	Let $L$, $M \in \mathbb{Z}_{>0}$, $J_\Omega$, $\hat{P}_\Omega$, $\hat{Q}_\Omega$ be defined as in \cref{lm:Mk}, \break $V \in \mathbb{C}^{m \times L}$, $S$ be defined as in~\eqref{eq:def_S}, and 
	\begin{align}
		Y = \left[ \hat{P}_\Omega V, J_\Omega \hat{P}_\Omega V, \dots, J_\Omega^{M-1} \hat{P}_\Omega V \right] \in \mathbb{C}^{m \times LM}.
	\end{align}
	Then, the equalities~$S = \hat{Q}_\Omega Y$ and $\mathrm{rank} (S) = \mathrm{rank} (Y)$ hold.
\end{lemma}

\begin{proof}
	From the definition~\cref{eq:def_S} of $S$ and \cref{lm:Mk}, it follows that
	\begin{align}
		S & = \left[ \mathsf{M}_0 V, \mathsf{M}_1 V, \dots, \mathsf{M}_{M-1} V \right] \\
		& = \left[ \hat{Q}_\Omega \hat{P}_\Omega V, \hat{Q}_\Omega J_\Omega \hat{P}_\Omega V, \dots, \hat{Q}_\Omega {J_\Omega}^{M-1} \hat{P}_\Omega V \right] \\
		& = \hat{Q}_\Omega \left[ \hat{P}_\Omega V, J_\Omega \hat{P}_\Omega V, \dots, {J_\Omega}^{M-1} \hat{P}_\Omega V \right] \\
		& = \hat{Q}_\Omega Y.
	\end{align}	
	Further, from \cite[Lemma~4.1]{YanaiTakeuchiTakane2011}, it follows that because $\mathcal{R} (Q_\Omega) \cap \mathcal{R} (Q_\mathrm{s}) = \lbrace \boldsymbol{0} \rbrace$, we have
	\begin{align}
		\mathrm{rank}(\Pi_{\mathcal{R}(Q_\mathrm{s})^\perp} Q_\Omega) & = \mathrm{rank}(Q_\Omega) \\
		& = t,
	\end{align}
	i.e., $\hat{Q}_\Omega = \Pi_{\mathcal{R}(Q_r)^\perp} Q_\Omega$ has full-column rank.
	Hence, $\mathrm{rank} (S) = \mathrm{rank} (Y)$ holds.
\end{proof}

Now, the main theorem can be proved.

\begin{theorem} \label{th:over}
	Let $L \in \mathbb{Z}_{>0}$, $V \in \mathbb{C}^{m \times L}$, $J_\Omega$, $\hat{P}_\Omega$, $\hat{Q}_\Omega \in \mathbb{C}^{n \times t}$ be defined as in \cref{lm:Mk}, $S$ be defined in~\eqref{eq:def_S}, and $t$ be the number of eigenvalues of matrix pencil~$z B - A$ counting multiplicity.	
	Then, $\mathrm{rank} (S) = t$ holds if and only if $\mathcal{R}(\hat{Q}_\Omega) = \mathcal{R}(S) \supseteq \mathcal{R}(X_\Omega)$ holds, where the columns of $X_\Omega$ are the eigenvectors corresponding to the eigenvalues in~$\Omega$.
\end{theorem}

\begin{proof}
	From \cref{lm:over}, it follows that $\mathcal{R} (S) = \mathcal{R} (\hat{Q}_\Omega Y) = \mathcal{R} (\hat{Q}_\Omega)$ holds if and only if $\mathrm{rank} (S) = t$ holds.
	The definitions of $\hat{Q}_\Omega$ and $X_\Omega$ give~$\mathcal{R}(\hat{Q}_\Omega) \supseteq \mathcal{R}(X_\Omega)$.
\end{proof}

\begin{remark}
	The desired eigenvector $\boldsymbol{x}$ can be obtained from $\boldsymbol{y}$ of \eqref{eq:TAS} by $\boldsymbol{x} = S \boldsymbol{y}$, where $\boldsymbol{y}$ is the eigenvector corresponding to an eigenvalue $\lambda$ of~\eqref{eq:TAS}.
\end{remark}

\begin{remark}
	The principle behind this formulation is the construction of a filter for eigencomponents.
	Combining \cref{lm:over} with \cref{th:over} shows that the orthogonal projector~$\Pi_{\mathcal{R}(Q_\mathrm{s})^\perp}$ in $S$ filters out the $\mathcal{R}(Q_\mathrm{s})$ component in the vector~$Q_\Omega Y \boldsymbol{y}$, i.e., stop undesired eigencomponents.
\end{remark}

\begin{remark}
	\Cref{th:over} shows that the numbers of $L$ and $M$ must be chosen to satisfy $L M \geq t$ to obtain the eigenvalues in $\Omega$.
	The number $t$ of eigenvalues in $\Omega$ would be estimated using analogous techniques given in \textrm{\cite{FutamuraTadanoSakurai2010, MaedaFutamuraImakuraSakurai2015JSIAMLetters, NapoliPolizziSaad2016, VecharynskiYang2017}}.
	High-order complex moments with~$M > 1$ enables one to efficiently enlarge the space~$\mathcal{R}(S)$.
\end{remark}

\begin{remark}
	The condition~$\mathrm{rank}(Y) = t$ in \cref{th:over} is not necessarily satisfied even if $\hat{P}_\Omega V$ has full-column rank.
\end{remark}

\subsection{Implementation}
In numerical computations, the contour integral~\eqref{eq:moment} is approximated using the $N$-point trapezoidal quadrature rule
\begin{align}
	\tilde{\mathsf{M}}_k = \sum_{j = 1}^N w_j z_j^k (z_j B - A)^\dag \simeq \mathsf{M}_k,
\end{align} 
where $z_j$ is a quadrature point and $\omega_j$ is its corresponding weight.
Thus, we obtain the approximations
\begin{align}
	\tilde{S}_k & = \tilde{\mathsf{M}}_k V, \quad \tilde{S} = \left[ \tilde{S}_0, \tilde{S}_1, \dots, \tilde{S}_{M-1} \right] \simeq S.
\end{align}
Moreover, to reduce computational costs and improve numerical stability, a low-rank approximation of the reduction~\eqref{eq:TAS} is applied using principal basis vectors of $\mathcal{R}(\tilde{S})$ associated with the truncated singular value decompositions (TSVD) of $\tilde{S}$
\begin{align}
	\tilde{S} & = U_\mathrm{S} \Sigma_\mathrm{S} {V_\mathrm{S}}^\mathsf{H} \\
	& = \left[ U_{\mathrm{S}, 1}, U_{\mathrm{S}, 2} \right] (\Sigma_{\mathrm{S}, 1} \oplus \Sigma_{\mathrm{S}, 2}) \left[ V_{\mathrm{S}, 1}, V_{\mathrm{S}, 2} \right]^\mathsf{H} \\
	& \simeq U_{\mathrm{S}, 1} \Sigma_{\mathrm{S}, 1} {V_{\mathrm{S}, 1}}^\mathsf{H},
	\label{eq:lowrankapprox}
\end{align}
where the columns of $U_\mathrm{S} = [U_{\mathrm{S}, 1}, U_{\mathrm{S}, 2}]$ are the left singular vectors of $\tilde{S}$, the columns of $V_\mathrm{S} = [V_{\mathrm{S}, 1}, V_{\mathrm{S}, 2}]$ are the right singular vectors of $\tilde{S}$, and $\Sigma_\mathrm{S} = \Sigma_{\mathrm{S}, 1} \oplus \Sigma_{\mathrm{S}, 2} \in \mathbb{R}^{n \times m}$ is a diagonal matrix whose diagonal entries are the singular values of $\tilde{S}$, arranged in decreasing order.
Here, the diagonal entries of $\Sigma_{\mathrm{S}, 1}$ are the dominating singular values of $\tilde{S}$, and the columns of $U_{\mathrm{S}, 1}$ and $V_{\mathrm{S}, 1}$ are the corresponding left and right singular vectors, respectively.
A practical way to determine the size~$\tau$ of $\Sigma_{S, 1} \in \mathbb{R}^{\tau \times \tau}$ will be given in Section~3.
Hence, $U_{\mathrm{{S}}, 1} \Sigma_{\mathrm{{S}}, 1} {V_{\mathrm{{S}}, 1}}^\mathsf{H}$ is a low-rank approximation of $\tilde{S}$.

Thus, by solving the reduced square generalized eigenproblem 
\begin{align}
	\tilde{T}^\mathsf{T} \! A U_{\mathrm{S}, 1} \tilde{\boldsymbol{y}} = \tilde{\lambda} \tilde{T}^\mathsf{T} \! B U_{\mathrm{S}, 1} \tilde{\boldsymbol{y}},
	\label{eq:prj_gevp}
\end{align}
where $\tilde{T} \in \mathbb{C}^{m \times \tau}$ is a random matrix, the approximate eigenpair $(\tilde{\lambda}, \tilde{\boldsymbol{x}}) = (\tilde{\lambda}, U_{\mathrm{S}, 1} \tilde{\boldsymbol{y}})$ can be obtained.
These procedures are summarized in \cref{alg}.

\begin{algorithm}
	\caption{Proposed method.}
	\label{alg}
	\begin{algorithmic}[1]
		\STATE Set $L$, $M$, $N \in \mathbb{Z}_{>0}$, $V \in \mathbb{C}^{m \times L}$, $(z_j, \omega_j)$, $j = 1, 2, \dots, N$.
		\STATE Compute $\tilde{S}_k = \sum_{j=1}^N \omega_j z_j^k (z_j B - A)^\dag V$; Set $\tilde{S} = [\tilde{S}_0, \tilde{S}_1, \dots, \tilde{S}_{M-1}]$.
		\STATE Compute SVD of $\tilde{S} = \left[ U_{\mathrm{S}, 1}, U_{\mathrm{S}, 2} \right] (\Sigma_{\mathrm{S}, 1} \oplus \Sigma_{\mathrm{S}, 2}) \left[ V_{\mathrm{S}, 1}, V_{\mathrm{S}, 2} \right]^\mathsf{H}$.
		\STATE Compute the eigenpairs~$(\tilde{\lambda}, \boldsymbol{y})$ of $\tilde{T}^\mathsf{T} \! A U_{\mathrm{S}, 1} \boldsymbol{y} = \tilde{\lambda} \tilde{T}^\mathsf{T} \! B U_{\mathrm{S}, 1} \boldsymbol{y}$.
		\STATE Compute the approximate eigenpairs $(\tilde{\lambda}, \tilde{x}) = (\tilde{\lambda}, U_{\mathrm{S}, 1} \tilde{\boldsymbol{y}})$.
	\end{algorithmic}
\end{algorithm}

\begin{remark}
	The computation of the pseudoinverse solution~$(z_j B - A)^\dag V$ in line 2 of \cref{alg} requires the largest cost.
	It is convenient that independent computations can be performed for each $j$ in parallel.
\end{remark}

Further, the pseudoinverse solution $(z_j B - A)^\dag V$ in line 2 of \cref{alg} may be efficiently computed by solving the minimum-norm least squares problem
\begin{align}
	& (z_j B - A)^\dag \boldsymbol{v}_i = \operatorname*{arg\,min}_{\boldsymbol{y} \in \mathbb{C}^n} \| \boldsymbol{y} \|_2, \label{eq:LS} \\
	& \qquad \text{subject to} \quad \min \| \boldsymbol{v}_i - (z_j B - A) \boldsymbol{y} \|_2, \quad i = 1, 2, \dots, L
\end{align}
by using (preconditioned) iterative solvers such as the CGLS, LSQR, LSMR, and AB- and BA-GMRES methods \cite{HestenesStiefel1952, PaigeSaunders1982a, FongSaunders2011, MorikuniHayami2013, MorikuniHayami2015}, where $\boldsymbol{v}_i$ is the $i$th column of the random matrix $V$ and $\| \cdot \|_2$ denotes the Euclidean norm, or the minimum-norm least squares problem with multiple right-hand side
\begin{align}
	(z_j B - A)^\dag V = \operatorname*{arg\,min}_{Y \in \mathbb{C}^{n \times L}} \| Y \|_\mathsf{F}, \quad \text{subject to} \quad \min \| V - (z_j B - A) Y \|_\mathsf{F}
\end{align}
by using iterative solvers such as the block CGLS and LSQR methods~\cite{JiLi2017, KarimiToutounian2006AMC} and the global CGLS method~\cite{LvHuangJiangLiu2015}, where $\| \cdot \|_\mathsf{F}$ denotes the Frobenius norm.
The $L$ psuedoinverse solutions~$(z_j B - A)^\dag \boldsymbol{v}_i$, $i= 1$, $2$  $\dots$, $L$, can be computed in parallel by solving the problems~\eqref{eq:LS}.

\section{Numerical experiments} \label{sec:nuex}
Numerical experiments show that the proposed method (\cref{alg}) is superior to previous methods in terms of efficiency and robustness.
The efficiency and robustness were evaluated in terms of CPU time and the relative residual norm (RRN)
\begin{align}
	\frac{\| A \boldsymbol{x} - \lambda B \boldsymbol{x} \|_2}{\| A \|_\mathsf{F} + |\lambda| \| B \|_\mathsf{F}}.
\end{align}

All computations were performed on a computer with an Intel Core i7-8565U 1.80 GHz central processing unit (CPU), 16 GB of  random-access memory (RAM), and the Microsoft Windows 10 Pro 64 bit Version 20H2 operating system.
All programs for implementing the proposed method were coded and run in MATLAB R2020b for double precision floating-point arithmetic with unit roundoff~$u = 2^{-53} \simeq 1.1\cdot10^{-16}$.
The compared implementation of GUPTRI was in the Matrix Canonical Structure (MCS) Toolbox~\cite{MCSToolbox}.

For the proposed method, the quadrature points $z_i$ were set to $z_i = \gamma + R \exp \left(\mathrm{i} \theta_i \right)$, $\theta_i = (2i-1) \pi / N $, $i = 1$, $2$, $\dots$, $N$, on the circle with center~$\gamma = 1 + \mathrm{i}$ and radius~$R$.
The MATLAB function~\texttt{pinv} was used to compute~$(z_j B - A)^\dag V$ for small sizes~$\min(m,  n) < 1000$.
On the other hand, the global CGLS method~\cite{JiLi2017} was used to compute it for large sizes~$\min(m, n) > 1000$ for efficiency.
Here, the initial guess was set zero and the stopping criteria for the global CGLS iteration were given by the threshold $10^{-14}$ in terms of the relative residual norm and the maximum number of iterations $\min(m, n)$.
In the row-rank approximation of $\tilde{S}$~\eqref{eq:lowrankapprox}, the rank was truncated such that the size of $\Sigma_\mathrm{S, 1}$ was maximized subject to the constraint on its condition number not exceeding $1/u$.
Here, the condition number of a matrix is the ratio of the largest singular value to the smallest value of the matrix.
The eigenpairs~$(\tilde{\lambda}, \boldsymbol{y})$ in line~4 of \cref{alg} were computed by using the MATLAB function~\texttt{eig}.

For comparison, the following folklore methods were tested.
Let $V \in \mathbb{C}^{n \times m}$ be a pseudo-random matrix. 
On the one hand, suppose $m < n$.
Then, the eigenpairs of square matrix pencil~$z BV - AV$ were computed by using the MATLAB function~\texttt{eig}.
This method is referred to as folklore method~1 herein.
Additionally, we computed the eigenpairs of the square matrix pencil
\begin{align}
	z
	\begin{bmatrix}
		B \\
		\mathrm{O}		
	\end{bmatrix}
	- 
	\begin{bmatrix}
		A \\
		\mathrm{O}
	\end{bmatrix}, \quad \mathrm{O} \in \mathbb{R}^{(n-m) \times n}
	\label{eq:zB-A;0}
\end{align}
using the MATLAB function~\texttt{eig}.
This method is referred to as folklore method~2.
On the other hand, suppose $m > n$.
Then, the eigenpairs of square matrix pencil~$z V B - V A$ were computed using the MATLAB function~\texttt{eig}.
This method is referred to as folklore method~3.
Further, the eigenpairs of square matrix pencil
\begin{align}
	z \left[ B, \mathrm{O} \right] - \left[ A, \mathrm{O}\right], \quad 
	\mathrm{O} \in \mathbb{R}^{m \times (m-n)}
	\label{eq:zB-A0}
\end{align}
were computed using the MATLAB function~\texttt{eig}.
This method is referred to as folklore method~4.
For further comparisons, we used the Rayleigh--Ritz-type contour integral\break method~\cite{IkegamiSakurai2010} for the above matrix pencils~$z BV - AV$ and $z V B - V A$, and call them folklore methods 1' and 3', respectively.
The matrix resolvent involved in this methods may not exist and was replaced with the pseudoinverse.
This pseudoinverse solution was computed similarly to the proposed method.
Applying the Rayleigh--Ritz-type contour integral method~\cite{IkegamiSakurai2010} to the above matrix pencils~\eqref{eq:zB-A;0} and \eqref{eq:zB-A0} is essentially the same as applying the proposed method to $z B - A$, when replacing the matrix resolvent involved in the former with the pseudoinverse.
Hence, we do not report on these approaches.

\subsection{Case~\texorpdfstring{$m < n$}{m < n}} \label{sec:under}
We illustrate experimental results for the case~$m < n$.
Test matrix pencils~$z B- A$ were generated as follows.
First, a matrix pencil $G(z) = z \mathrm{I}_\eta - \Lambda \in \mathbb{C}^{\eta \times \eta}$ having finite eigenvalues was generated. Here, $\Lambda \in \mathbb{C}^{\eta \times \eta}$ is a diagonal matrix whose diagonal entries are complex numbers having real and imaginary parts drawn from the standardized normal distribution.
Second, a matrix pencil $K(z) = z \oplus_{i=1}^p~N_{\rho_i} - \mathrm{I}_\rho \in \mathbb{C}^{\rho \times \rho}$ having infinite eigenvalues was generated. Here, $N_{\rho_i}$ is a shift matrix whose size was randomly chosen such that $\rho / 2$ superdiagonal entries of $\oplus_{i=1}^p~N_{\rho_i}$ are one and the remaining entries are zero.
Finally, the test matrices
\begin{align}
	A & = R_1
	\begin{bmatrix} 
		\Lambda \\
		& \mathrm{I}_\rho \\
		& & \bigzerol
	\end{bmatrix} R_2
	\in \mathbb{C}^{m \times n}, \label{eq:test_A} \\	
	B & = R_1
	\begin{bmatrix} 
		\mathrm{I}_\eta \\ 
		& \oplus_{i=1}^p~N_{\rho_i} \\
		& & \bigzerol
	\end{bmatrix} R_2
	\in \mathbb{C}^{m \times n},
	\label{eq:test_B}
\end{align}
were formed, where $R_1$ and $R_2$ were generated using the MATLAB function~\texttt{randn} for $(m, n) = (30, 100)$, $(300, 1000)$; the products of randomly generated Givens rotation matrices were used to set the nonzero density of sparse matrices~$A$ and $B$ to $0.001$ for $(m, n) = (3000, 10000)$, $(30000, 100000)$.
The block sizes~$\eta = \rho = r$, $q = n - \eta - \rho$, and $\mu = \nu = 0$ were set.

\Cref{tbl:info_m<n} gives information on the test matrix pencils. This includes the size of each matrix pencil and its Kronecker blocks, and the values of parameters for the proposed method, including the center and radius of the circle~$\Gamma$ of the prescribed region~$\Omega$, number of eigenvalues $t$ in the region, number of columns $L$ of matrix~$V$, order of complex moments~$M$, and number of quadrature points~$N$.

\begin{table}[t]
	\footnotesize	
	\centering
	\caption{Parameter values for $m < n$.}
	\begin{tabular}{rrrrrrrrrr}
		\toprule
		$m$ &         $n$ &     $\eta$ &     $\rho$ &        $q$ &    $R$ & $t$ & $L$ & $M$ &  $N$ \\ \midrule
		$30$ &       $100$ &       $10$ &       $10$ &       $10$ &    $1$ & $2$ & $4$ & $2$ & $48$ \\
		$300$ &   $1{,}000$ &      $100$ &      $100$ &      $100$ &  $0.3$ & $3$ & $4$ & $2$ & $48$ \\
		$3{,}000$ &  $10{,}000$ &  $1{,}000$ &  $1{,}000$ &  $1{,}000$ &  $0.1$ & $3$ & $8$ & $4$ & $48$ \\
		$30{,}000$ & $100{,}000$ & $10{,}000$ & $10{,}000$ & $10{,}000$ & $0.05$ & $2$ & $8$ & $4$ & $48$ \\ \bottomrule
	\end{tabular}
	\begin{minipage}{1\hsize}
		$m$:~number of rows of a matrix, $n$:~number of columns of a matrix, $\eta$:~size of the Kronecker block corresponding to finite eigenvalues, $\rho$:~size of the Kronecker block corresponding to infinite eigenvalues, $R$:~radius, $t$:~number of eigenvalues in the curve~$\Gamma$, $L$:~number of columns of the pseudo-random matrix~$V$, $M$: order of complex moments, and $N$: number of quadrature points.
	\end{minipage}
	\label{tbl:info_m<n}
	\medskip
	
	\caption{CPU time (s) and maximum relative residual norm for $m < n$.}
	\begin{tabular}{c@{\hskip 24pt}rrr@{\hskip 24pt}rrr}
		\toprule
		$(m, n)$ &                \multicolumn{3}{c}{$(30, 100)$}  &               \multicolumn{3}{c}{$(300,1000)$} \\
		&    time &       $\max$ RERR &        $\max$ RRN &   time &       $\max$ RERR &        $\max$ RRN \\
		\midrule
		folklore~1  & * 0.003 & \texttt{1.64e-14} & \texttt{1.31e-15} & * 0.39 & \texttt{1.58e-13} & \texttt{8.98e-15} \\
		folklore~1' &   0.026 & \texttt{1.95e-15} & \texttt{1.70e-16} &   2.27 & \texttt{3.40e-14} & \texttt{8.60e-16} \\
		folklore~2  & * 0.003 & \texttt{8.95e-15} & \texttt{9.98e-17} &   0.92 & \texttt{8.77e-15} & \texttt{2.73e-16} \\
		GUPTRI   	&   0.110 & \texttt{7.74e-15} &               --- &   7.66 & \texttt{7.55e-15} &               --- \\
		proposed 	&   0.062 & \texttt{5.48e-15} & \texttt{5.24e-16} &   3.52 & \texttt{3.20e-14} & \texttt{1.99e-15} \\
		\bottomrule
	\end{tabular}
	
	\begin{tabular}{c@{\hskip 24pt}rrr@{\hskip 24pt}rrr}
		\toprule
		$(m, n)$ &              \multicolumn{3}{c}{$(3000, 10000)$} &               \multicolumn{3}{c}{$(30000, 100000)$} \\
		&     time &       $\max$ RERR &        $\max$ RRN &      time &       $\max$ RERR &          $\max$ RRN \\
		\midrule
		folklore~1  &    228.6 & \texttt{1.46e-14} & \texttt{6.12e-14} &    $>$24h &               --- &                 --- \\
		folklore~1' & 17{,}930 & \texttt{2.03e-14} & \texttt{3.15e-16} &    $>$24h &               --- &                 --- \\
		folklore~2  &    159.8 &               --- &               --- &    $>$24h &               --- &                 --- \\
		GUPTRI      &  7{,}503 & \texttt{9.66e-14} &               --- &       --- &               --- &                 --- \\
		proposed    &  * 77.84 & \texttt{7.83e-15} & \texttt{5.12e-16} & * 8{,}372 & \texttt{2.01e-14} & \texttt{3.82e-16}   \\
		\bottomrule
	\end{tabular}
	\begin{minipage}{1\hsize}
		$m$: number of rows of a matrix, $n$: number of columns of a matrix, time:~elapsed CPU time (s), $\max$ RERR:~maximum relative error, $\max$ RRN:~maximum relative residual norm.
	\end{minipage}
	\label{tbl:resnrm_time_m<n}
\end{table}

\Cref{tbl:resnrm_time_m<n} gives the elapsed CPU time of each method in seconds, as well as the maximum relative error (max RERR) of eigenvalues in the prescribed region and the maximum RRN (max RRN) in terms of eigenvalues in the prescribed region and the corresponding eigenvectors.
The symbol~$*$ indicates the least CPU time among the compared methods for each test problem.
\Cref{tbl:resnrm_time_m<n} shows that the proposed method was faster than GUPTRI in terms of CPU time and was more robust than other methods.
Folklore methods~1, 1', and 2 were faster than the proposed method for the cases~$(m, n) = (30, 100)$, $(300, 1000)$; however they did not terminate within 24 hours for the case~$(m, n) = (30000, 100000)$.
Folklore method~2 computed only one of the three eigenvalues in the prescribed region for the case~$(m, n) = (3000, 10000)$.
GUPTRI fails to give finite eigenvalues for the case~$(m, n) = (30000, 100000)$.
This may be due to rounding errors.
Note that GUPTRI does not compute the eigenvectors and hence the maximum RNN for the method is not given.

To examine the effect of the number of quadrature points $N$ on the accuracy of the proposed method, we tested the method on different numbers of $N = 8$, $9$, $\dots$, $48$.
Here, the values of the other parameters were the same as those given in \cref{tbl:info_m<n}.
\Cref{fig:N_vs_maxRRN} shows the maximum RERR and RRN and CPU time in seconds versus the number of quadrature points~$N$ on the above test matrix pencil with $(m, n) = (3000, 10{,}000)$.
This figure shows that the maximum RERR and RRN seem to exponentially converge to the machine epsilon regarding $N$, and the CPU time increased proportionally to $N$.
The Supplementary Material describes the relationship between the quadrature error and the number of quadrature points in a diagonalizable case.
We see that the number of quadrature points~$N = 48$ chosen for the above tests is sufficiently large and taken as a safeguard.
Similar trends were observed for the other problems in this study.
However, it is not clear how to find the least number of quadrature points~$N$ required to achieve a specified accuracy in general.

\Cref{tbl:proportion} gives the proportion of CPU time in seconds for each step of \cref{alg} with quadrature points~$N = 48$ on the above test matrix pencils.
The CPU time for step~2 occupies more than 98 \% of the total CPU time.
As the matrix size increases, the proportion of step~2 increases and approaches one.
The experiments were performed on a serial computer; the elapsed CPU time would be approximately $1/N$ or $1/(LN)$ (cf.~\eqref{eq:LS}) for large problems when the solutions of the (block) least squares problems in step~2 of \cref{alg} are implemented in parallel.

\subsection{Case~\texorpdfstring{$m > n$}{m > n}} \label{sec:over}
We show experimental results for the case~$m > n$.
Test matrix pencils~$z B- A$ were generated as in \eqref{eq:test_A} and \eqref{eq:test_B} with block sizes~$\eta = \rho = q$, $r = m - \eta - \rho$, and $\mu = \nu = 0$.
As in \cref{tbl:info_m<n}, \cref{tbl:info_m>n} gives information on the test matrix pencils.

\Cref{tbl:resnrm_time_m>n} gives the elapsed CPU time of each method, as well as the maximum RERR and RRN, similarly to \cref{tbl:resnrm_time_m<n}.
The table shows that the proposed method was faster than GUPTRI in terms of CPU time and was more robust than other methods.
Folklore methods~3 and 4 were faster than the proposed method for the cases~$(m, n) = (100, 30)$, $(1000, 300)$; however they took more than 24 hours for the case~$(m, n) = (100000, 30000)$.
Folklore method~3 did not give eigenvalues in the prescribed region for the case~$(m, n) = (10000, 3000)$.
Folklore method~4 computed only two of the three eigenvalues in the prescribed region for the case~$(m, n) = (10000, 3000)$.
GUPTRI failed to give finite eigenvalues for $(m, n) = (100000, 30000)$.
This may be due to rounding errors.

{
	\begin{figure}
		\centering
		\includegraphics[width=0.5\linewidth]{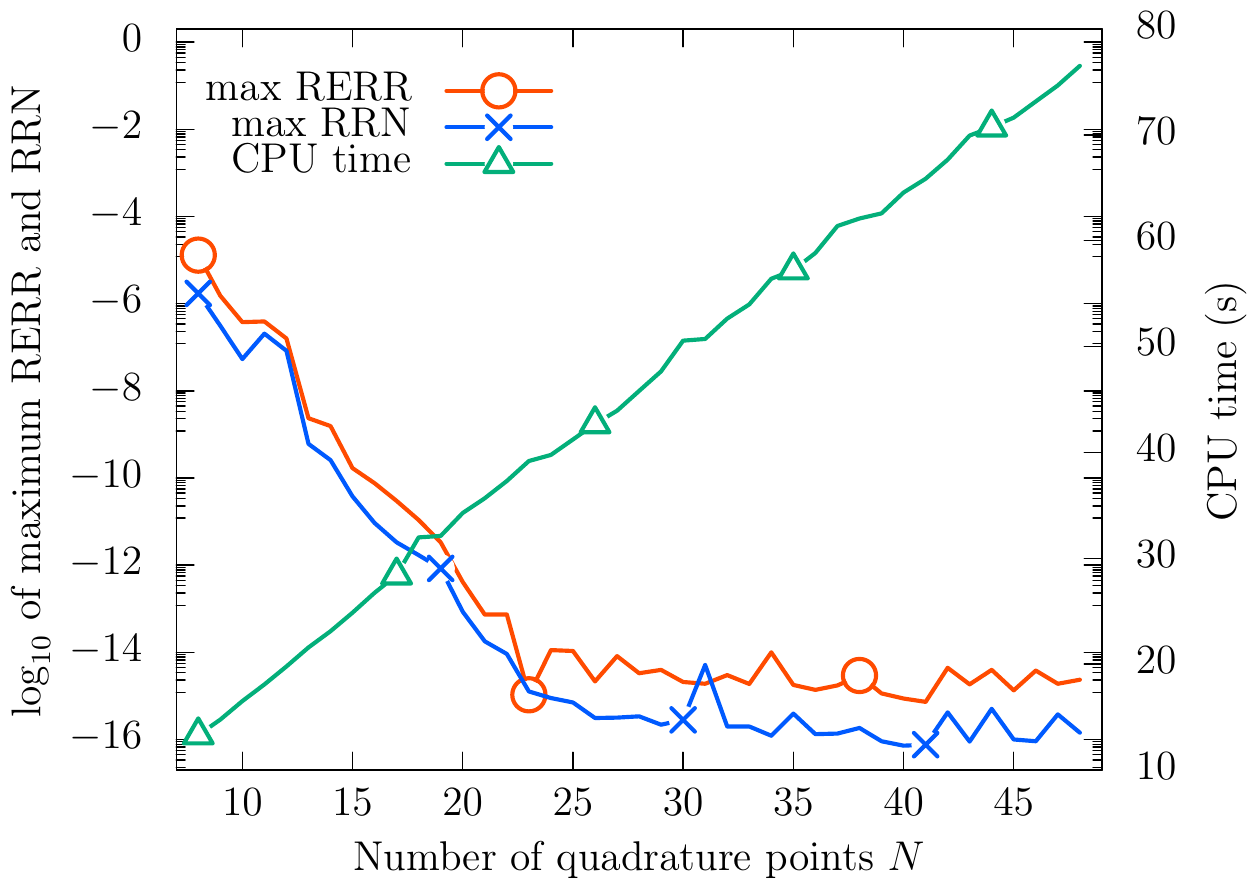}
		\caption{Maximum relative error, maximum relative residual norm, and CPU time (s) versus the number of quadrature points~$N$.}
		\label{fig:N_vs_maxRRN}
	\end{figure}
	\begin{table}
		\footnotesize
		\centering
		\caption{Proportion (\%) of the CPU time (s) for each step of \cref{alg}.}
		\begin{tabular}{rr@{\hskip 24pt}rrrrrrrr}
			\toprule
			$m$ &         $n$ & \multicolumn{5}{c}{Step}\\
			&             &          1 &      2 &      3 &      4 &      5 \\ \midrule
			$30$ &       $100$ & $0.65$ & $98.38$ & $0.40$ & $0.48$ & $0.10$ \\
			$300$ &   $1{,}000$ & $0.01$ & $99.93$ & $0.01$ & $0.05$ & $0.01$ \\
			$3{,}000$ &  $10{,}000$ & $0.00$ & $99.94$ & $0.03$ & $0.03$ & $0.01$ \\
			$30{,}000$ & $100{,}000$ & $0.00$ & $99.98$ & $0.01$ & $0.01$ & $0.00$ \\ \bottomrule
		\end{tabular}
		\label{tbl:proportion}
		\begin{minipage}{1\hsize}
			$m$: number of rows of a matrix, $n$: number of columns of a matrix, step: step number of \cref{alg}.
		\end{minipage}
	\end{table}
}

\begin{table}[t]
	\footnotesize
	\centering
	\caption{Parameter values for $m > n$.}
	\begin{tabular}{rrrrrrrrrr}
		\toprule
		$m$ &        $n$ &     $\eta$ &     $\rho$ &        $r$ &    $R$ & $t$ & $L$ & $M$ &  $N$ \\ \midrule
		$100$ &       $30$ &       $10$ &       $10$ &       $10$ &    $1$ & $2$ & $4$ & $2$ & $48$ \\
		$1{,}000$ &      $300$ &      $100$ &      $100$ &      $100$ &  $0.3$ & $3$ & $4$ & $2$ & $48$ \\
		$10{,}000$ &  $3{,}000$ &  $1{,}000$ &  $1{,}000$ &  $1{,}000$ &  $0.1$ & $3$ & $8$ & $4$ & $48$ \\
		$100{,}000$ & $30{,}000$ & $10{,}000$ & $10{,}000$ & $10{,}000$ & $0.05$ & $2$ & $8$ & $4$ & $48$ \\ \bottomrule
	\end{tabular}
	\begin{minipage}{1\hsize}
		$m$:~number of rows of a matrix, $n$:~number of columns of a matrix, $\eta$:~size of the Kronecker block corresponding to finite eigenvalues, $\rho$:~size of the Kronecker block corresponding to infinite eigenvalues, $R$:~radius, $t$:~number of eigenvalues in the curve~$\Gamma$, $L$:~number of columns of the pseudo-random matrix~$V$, $M$: order of complex moments, and $N$: number of quadrature points.
	\end{minipage}
	\label{tbl:info_m>n}
	\medskip
	
	\caption{CPU time (s) and maximum relative residual norm for $m > n$.}
	\begin{tabular}{l@{\hskip 10pt}rrr@{\hskip 24pt}rrr}
		\toprule
		\multicolumn{1}{c}{$(m, n)$}&\multicolumn{3}{c}{$(100, 30)$}  &              \multicolumn{3}{c}{$(1000, 300)$} \\
		&    time &       $\max$ RERR &        $\max$ RRN &   time &       $\max$ RERR &        $\max$ RRN \\
		\midrule
		folklore~3  & * 0.003 & \texttt{1.47e-12} & \texttt{3.28e-16} & * 0.58 & \texttt{5.07e-15} & \texttt{5.78e-16} \\
		folklore~3' &   0.023 & \texttt{4.69e-15} & \texttt{5.46e-16} &   1.67 & \texttt{5.11e-14} & \texttt{3.02e-15} \\
		folklore~4  &   0.007 & \texttt{1.61e-12} & \texttt{3.55e-16} &   1.04 & \texttt{4.83e-15} & \texttt{7.46e-16} \\
		GUPTRI      &   0.083 & \texttt{6.08e-15} &               --- &   5.66 & \texttt{8.83e-15} &               --- \\
		proposed 	&   0.021 & \texttt{6.20e-15} & \texttt{1.96e-15} &   2.46 & \texttt{3.99e-15} & \texttt{4.64e-16} \\
		\bottomrule
	\end{tabular}
	\medskip
	
	\begin{tabular}{l@{\hskip 24pt}rrr@{\hskip 24pt}rrr}
		\toprule
		\multicolumn{1}{c}{$(m, n)$}&\multicolumn{3}{c}{$(10000, 3000)$} &          \multicolumn{3}{c}{$(100000, 30000)$} \\
		&     time &       $\max$ RERR &        $\max$ RRN &      time &       $\max$ RERR &        $\max$ RRN \\ 
		\midrule
		folklore~3  &    188.4 &               --- &               --- &    $>$24h &               --- &               --- \\
		folklore~3' & 17{,}790 & \texttt{4.19e-15} & \texttt{6.63e-16} &    $>$24h &               --- &	              --- \\
		folklore~4  &    407.2 &               --- &               --- &    $>$24h &               --- &               --- \\
		GUPTRI      &  7{,}777 & \texttt{5.71e-14} &               --- &       --- &               --- &               --- \\
		proposed    &  * 72.00 & \texttt{7.64e-16} & \texttt{1.61e-16} & * 8{,}070 & \texttt{5.41e-15} & \texttt{5.00e-16} \\ 
		\bottomrule
	\end{tabular}
	\begin{minipage}{1\hsize}
		$m$: number of rows of a matrix, $n$: number of columns of a matrix, time:~elapsed CPU time (s), $\max$ RERR:~maximum relative error, $\max$ RRN:~maximum relative residual norm.
	\end{minipage}
	\label{tbl:resnrm_time_m>n}
\end{table}

\subsection{Case~\texorpdfstring{$\nu > 0$}{nu > 0}} \label{sec:over_nu}
We show experimental results for the case~$\nu > 0$, although the proposed method is not theoretically guaranteed to work in this case.
Test matrix pencils~$z B- A \in \mathbb{C}^{m \times n}$ were generated as follows:
\begin{align}
	A & = R_1
	\begin{bmatrix} 
		\Lambda & & \bigzerou \\
		& \mathrm{I}_\rho \\
		& & {N_\nu}^\mathsf{T} \\
		\bigzerol & & \boldsymbol{e}_\nu^\mathsf{T}
	\end{bmatrix} R_2
	\in \mathbb{C}^{m \times n}, \\	
	B & = R_1
	\begin{bmatrix} 
		\mathrm{I}_\eta & & \bigzerou \\ 
		& \oplus_{i=1}^p~N_{\rho_i} \\
		& & \mathrm{I}_\nu \\
		\bigzerol & & \boldsymbol{0}^\mathsf{T}
	\end{bmatrix} R_2
	\in \mathbb{C}^{m \times n},
\end{align}
where $\Lambda$, $\oplus_{i=1}^p~N_{\rho_i}$, $R_1$, and $R_2$ were generated as in \cref{sec:under}.
We fixed $m = 10000$, $\eta = \rho = 1000$, $r = 1$, and $R = 0.1$, and varied $\nu = 1$, $10$, $100$, and $1000$ ($n = 2001$, $2010$, $2100$, and $3000$, respectively).
Then, the number of eigenvalues in the prescribed region was $t = 3$.
The numbers of parameters were set to $L = 8$, $M = 4$, and $N = 48$.

\Cref{tbl:resnrm_time_nu} gives the elapsed CPU time of each method and the maximum RERR and RRN, similarly to \cref{tbl:resnrm_time_m<n}.
The table shows that the proposed method was faster than other methods in terms of the CPU time. 
Folklore method~4 computed only two of the three eigenvalues in the prescribed region for the case~$\nu = 1000$.
GUPTRI did not give the three eigenvalues in the prescribed region for $\nu = 10$ and $100$.

\begin{table}[t]
	\footnotesize
	\centering
	\caption{CPU time (s) and maximum relative residual norm for varied $\nu > 0$.}
	\begin{tabular}{l@{\hskip 24pt}rrr@{\hskip 24pt}rrr}
		\toprule
		\multicolumn{1}{c}{$\nu$}  &                      \multicolumn{3}{c}{$1$} &                         \multicolumn{3}{c}{$10$} \\
		&     time &       $\max$ RERR &        $\max$ RRN &     time &       $\max$ RERR &        $\max$ RRN \\ 
		\midrule
		folklore~3  &    122.8 & \texttt{5.34e-15} & \texttt{2.05e-15} &    119.5 & \texttt{9.69e-15} & \texttt{9.13e-15} \\
		folklore~3' & 81{,}312 & \texttt{3.64e-14} & \texttt{2.00e-13} &   $>$24h &               --- &               --- \\
		folklore~4  &    336.1 & \texttt{1.01e-14} & \texttt{1.20e-16} &    377.8 & \texttt{4.07e-14} & \texttt{1.49e-16} \\
		GUPTRI      &  7{,}354 & \texttt{7.90e-14} &               --- & 14{,}525 &               --- &               --- \\
		proposed    &  * 66.97 & \texttt{2.77e-15} & \texttt{2.16e-16} &  * 65.20 & \texttt{3.12e-15} & \texttt{2.04e-16} \\
		\bottomrule
	\end{tabular}
	\medskip
	
	\begin{tabular}{l@{\hskip 24pt}rrr@{\hskip 24pt}rrr}
		\toprule
		\multicolumn{1}{c}{$\nu$} &          \multicolumn{3}{c}{$100$} &         \multicolumn{3}{c}{$1000$}    \\
		&     time &       $\max$ RERR &        $\max$ RRN &    time &       $\max$ RERR &        $\max$ RRN \\
		\midrule
		folklore~3  &    118.4 & \texttt{1.62e-14} & \texttt{3.77e-15} &   317.7 & \texttt{7.35e-15} & \texttt{1.20e-14} \\
		folklore~3' &   $>$24h &               --- &               --- &  $>$24h &               --- &               --- \\
		folklore~4  &    349.1 & \texttt{5.93e-02} & \texttt{1.64e-16} &   967.3 &               --- &               --- \\
		GUPTRI      & 19{,}389 &               --- &               --- &  $>$24h &               --- &               --- \\
		proposed    &  * 67.94 & \texttt{5.16e-15} & \texttt{2.26e-16} & * 73.20 & \texttt{9.62e-15} & \texttt{4.31e-16} \\	
		\bottomrule		
	\end{tabular}
	\begin{minipage}{1\hsize}
		$\nu$:~number of columns of the left singular block, time:~elapsed CPU time (s), $\max$ RERR:~maximum relative error, $\max$ RRN:~maximum relative residual norm.
	\end{minipage}
	\label{tbl:resnrm_time_nu}
\end{table}

\section{Conclusions} \label{sec:conc}
In this study, a projection method for computing interior eigenvalues and corresponding eigenvectors of square linear matrix pencils was extended to nonsquare cases.
The proposed method was successfully demonstrated on matrix pencils with specific Kronecker structures.
Numerical experiments showed that this approach is superior to previous ones in terms of efficiency and robustness for large problems.
Experimental results conjecture that the proposed method works on matrix pencils with left singular blocks and theoretical support for this case was left open.
This extension took the Rayleigh--Ritz-type approach and may be applied to other types of projection methods~\cite{SakuraiSugiura2003JCAM,Polizzi2009,Beyn2012LAA,ImakuraDuSakurai2014}.
This study provides direction for tackling further extensions of the projection method to nonlinear nonsquare matrix eigenproblems, cf.~\cite{AsakuraSakuraiTadanoIkegamiKimura2009JSIAMLett,AsakuraSakuraiTadanoIkegamiKimura2010JJIAM,Beyn2012LAA,YokotaSakurai2013JSIAMLett}.

\section*{Acknowledgments}
The author would like to thank Professor~Ken Hayami, Professor~Tetsuya Sakurai, Doctor~Akira Imakura, and Doctor~Ning Zheng for their valuable comments.
We would like to thank the referees for their valuable comments.

\end{document}